\documentclass{sl}     
\usepackage{slsec}     
\usepackage{stud_log}
\usepackage{slfoot}
\usepackage{slthm}     

\usepackage[utf8]{inputenc}
\usepackage{xcolor}

\usepackage{amsmath}
\usepackage{amssymb}

\usepackage{enumerate}

\newcommand{\axiomsystem}[1]{\mathrm{#1}}
\newcommand{\C}[0]{\axiomsystem{C}}
\newcommand{\D}[0]{\axiomsystem{D}}
\newcommand{\logic}[1]{\mathsf{#1}}
\newcommand{\dtwo}[0]{\logic{D2}}
\newcommand{\Sfive}[0]{\logic{S5}}
\newcommand{\ddk}[1]{\mathrm{D}_{#1}}
\newcommand{\ciuciura}[1]{\mathrm{C}_{#1}}
\newcommand{\Cunprovable}[1]{\not\vdash_{\mathrm{C}}{#1}}
\newcommand{\Dunprovable}[1]{\not\vdash_{\mathrm{D}}{#1}}
\newcommand{\negationsymbol}[0]{\mathop\sim}
\newcommand{\implicationsymbol}[0]{\rightarrow}
\newcommand{\conjunctionsymbol}[0]{\wedge}
\newcommand{\discussiveimplicationsymbol}[0]{\implicationsymbol_{\mathrm{d}}}
\newcommand{\discussiveconjunctionsymbol}[0]{\conjunctionsymbol_{\mathrm{d}}}
\newcommand{\disjunctionsymbol}[0]{\vee}
\newcommand{\di}[2]{{#1} \discussiveimplicationsymbol {#2}}
\newcommand{\dc}[2]{{#1} \discussiveconjunctionsymbol {#2}}
\newcommand{\disjunction}[2]{{#1} \disjunctionsymbol {#2}}
\newcommand{\conjunction}[2]{{#1} \conjunctionsymbol {#2}}
\newcommand{\negation}[1]{\negationsymbol {#1}}
\newcommand{\implication}[2]{{#1} \implicationsymbol {#2}}
\newcommand{\equivalencesymbol}[0]{\leftrightarrow}
\newcommand{\equivalence}[2]{{#1} \equivalencesymbol {#2}}
\newcommand{\possibly}[1]{\lozenge{#1}}
\newcommand{\falsifies}[2]{{#1} \nvDash {#2}}
\newcommand{\truthifies}[2]{{#1} \vDash {#2}}

\def\kn{\kern.1em}

\newtheorem{theo}{Theorem}[section]
\newtheorem{coro}[theo]{Corollary}
\newtheorem{prop}[theo]{Proposition}

\theoremstyle{definition}

\HeadingsInfo{Alama}{Some problems with two axiomatizations of discussive logic}

\begin{document}

\setcounter{page}{1}     



\AuthorTitle{Jesse Alama}{Some problems with two axiomatizations of discussive logic}






\PresentedReceived{Name of Editor}{December 1, 2005}


\begin{abstract}Problems in two axiomatizations of Jaśkowski's discussive (or discursive) logic $\dtwo$ are considered.  A recent axiomatization of $\dtwo$ and completeness proof relative to $\dtwo$'s intended semantics seems to be mistaken because some formulas valid according to the intended semantics turn out to be unprovable.  Although no new axiomatization is offered, nor a repaired completeness proof given, the shortcomings identified here may be a step toward an improved axiomatization.\end{abstract}

\Keywords{discussive logic, discursive logic, modal logic, automated theorem proving, countermodel, matrices
}


\section*{Introduction}

Jaśkowski's discussive (or discursive) logic is a paraconsistent logic of opinions.  The idea is that it is one thing is a reasoner contradicts himself; it is quite another for two reasoners to contradict one another.  The latter is an everyday phenomenon that is all too familiar, but difficult to justify from the standard logical standpoint, which accepts ``explosion principles'' such as ex contradictione quodlibet ($\implication{(\conjunction{p}{\negation{p}})}{q}$.  The idea is that an opinion, on its own, is to be understood as consistent (in the logical sense of the term), but might conflict with other opinions.  It is one thing if a speaker contradicts himself; it is another if two speakers hold conflicting opinions.  We recognize the former as leading to a kind of incoherence; the latter is not necessarily incoherent, but perhaps calls for resolution of conflict.

Jaśkowski did not offer an axiomatization of his new logic.  His contribution was to formulate the problem of giving a rigorous account of a certain paraconsistent discursive phenomenon.  He spent much of his attention on a discussion of the notion of paraconsistency and its justification from the standard non-paraconsistent point of view.  Axiomatizations came later.

In a recent paper~\cite{ciuciura2008frontiers}, Ciuciura takes issue with an axiomatization, $\D$ (due to da~Costa, Dubikajtis, and Kotas~\cite{achtelik1981independence,kotas1977modal,dacosta1977new}) of $\dtwo$.  He points out two axioms of $\D$ that are not valid in the intended semantics of $\dtwo$.  To repair the difficulties, Ciuciura gave a new axiomatization, $\C$, of $\dtwo$ and proved soundness and completeness of the new axiomatization relative to the standard semantics of $\dtwo$, which is based on the modal logic $\Sfive$.

Although $\D$ apparently does go beyond the intention of $\dtwo$ (that is, there are theorems of $\D$ that are not valid in the intended semantics of $\dtwo$), Ciuciura's proposed fix for the problem (that is, the replacement of $\D$ by the new axiom system $\C$) is itself problematic.  Namely, there are theses in the language of $\dtwo$ that are valid according to the intended semantics but are not theorems of $\C$.  A consequence of this is that Ciuciura's completeness proof is somehow flawed.

In Section~\ref{sec:dueling} the two axiomatizations $\C$ of Ciuciura and $\D$ of da~Costa, Dubikajtis, and Kotas are presented.  Section~\ref{sec:non-equivalence} treats the non-equivalence of $\C$ and $\D$.  Ciuciura pointed out two axioms of $\D$ that fail in the intended semantics of $\dtwo$.  Although even one such example suffices to make his point, we complete the discussion by pointing out even more axioms of $\D$, not mentioned by Ciuciura, that are likewise invalid.  One might get the impression from Ciuciura's criticism that $\D$ merely `overshoots' the intended semantics and can be repaired by simply dropping some axioms.  We find, though, that $\C$ is not simply a restriction of $\D$: there are axioms of $\C$ that are not theorems of $\D$.  We are left in a rather puzzling situation: $\C$ and $\D$ are ``orthogonal'' in the sense that they overlap one another, but each has theorems that are not found in the other.  Moreover, we find that there are some $\dtwo$-valid formulas unprovable in $\D$.  It would appear that the problem of axiomatizing $\dtwo$ remains unsolved.

\section{Dueling axiomatizations}\label{sec:dueling}%

For the sake of completeness, below are two axiomatizations in the language $\negationsymbol$ (negation), $\disjunctionsymbol$ (disjunction), $\discussiveimplicationsymbol$ (discussive implication), and $\discussiveconjunctionsymbol$ (discussive conjunction).  The first batch is due to da~Costa, Dubikajtis, and Kotas and is, for short, called simply ``$\D$''; the second batch is due to Ciuciura and is called ``$\C$''.  Both $\C$ and $\D$ use modus ponens (for discussive implication) as their sole rule of inference.  The axiom lists are simply repeated from~\cite{ciuciura2008frontiers} (to more clearly distinguish the two axiom systems, we use the ``$\C$'' and ``$\D$'' prefixes, whereas Ciuciura used the prefix ``$\mathrm{A}$'' for both sets of axioms).  The axiom sets are given schematically; Greek letters are variables ranging over formulas.

\begin{description}
\item[$\ddk{1}$] $\di{\alpha}{(\di{\beta}{\alpha})}$
\item[$\ddk{2}$] $\di{(\di{\alpha}{(\di{\beta}{\gamma})})}{(\di{(\di{\alpha}{\beta})}{(\di{\alpha}{\gamma})})}$
\item[$\ddk{3}$] $\di{(\di{(\di{\alpha}{\beta})}{\alpha})}{\alpha}$
\item[$\ddk{4}$] $\di{(\dc{\alpha}{\beta})}{\alpha}$
\item[$\ddk{5}$] $\di{(\dc{\alpha}{\beta})}{\beta}$
\item[$\ddk{6}$] $\di{\alpha}{(\di{\beta}{(\dc{\alpha}{\beta})})}$
\item[$\ddk{7}$] $\di{\alpha}{(\disjunction{\alpha}{\beta})}$
\item[$\ddk{8}$] $\di{\beta}{(\disjunction{\alpha}{\beta})}$
\item[$\ddk{9}$] $\di{(\di{\alpha}{\gamma})}{(\di{(\di{\beta}{\gamma})}{\di{(\disjunction{\alpha}{\beta})}{\gamma}})}$
\item[$\ddk{10}$] $\di{\alpha}{\negation{\negation{\alpha}}}$
\item[$\ddk{11}$] $\di{\negation{\negation{\alpha}}}{\alpha}$
\item[$\ddk{12}$] $\di{\negation{(\disjunction{\alpha}{\negation{\alpha}})}}{\beta}$
\item[$\ddk{13}$] $\di{\negation{(\disjunction{\alpha}{\beta})}}{\negation{(\disjunction{\beta}{\alpha})}}$
\item[$\ddk{14}$] $\di{\negation{(\disjunction{\alpha}{\beta})}}{(\dc{\negation{\beta}}{\negation{\alpha}})}$
\item[$\ddk{15}$] $\di{\negation{(\disjunction{\negation{\negation{\alpha}}}{\beta})}}{\negation{(\disjunction{\alpha}{\beta})}}$
\item[$\ddk{16}$] $\di{(\negation{\di{(\disjunction{\alpha}{\beta})}{\gamma}})}{(\disjunction{(\di{\negation{\alpha}}{\beta})}{\gamma})}$
\item[$\ddk{17}$] $\di{\negation{(\disjunction{(\disjunction{\alpha}{\beta})}{\gamma})}}{\negation{(\disjunction{\alpha}{(\disjunction{\beta}{\gamma})})}}$
\item[$\ddk{18}$] $\di{\negation{(\disjunction{(\di{\alpha}{\beta})}{\gamma})}}{(\dc{\alpha}{\negation{(\disjunction{\beta}{\gamma})}})}$
\item[$\ddk{19}$] $\di{\negation{(\disjunction{(\dc{\alpha}{\beta})}{\gamma})}}{(\di{\alpha}{\negation{(\disjunction{\beta}{\gamma})}})}$
\item[$\ddk{20}$] $\di{\negation{(\disjunction{\negation{(\disjunction{\alpha}{\beta})}}{\gamma})}}{(\disjunction{\negation{(\disjunction{\negation{\alpha}}{\gamma})}}{\negation{(\disjunction{\negation{\beta}}{\gamma})}})}$
\item[$\ddk{21}$] $\di{\negation{(\disjunction{\negation{(\di{\alpha}{\beta})}}{\gamma})}}{(\di{\alpha}{\negation{(\disjunction{\negation{\beta}}{\gamma})}})}$
\item[$\ddk{22}$] $\di{\negation{(\disjunction{\negation{(\dc{\alpha}{\beta})}}{\gamma})}}{(\dc{\alpha}{\negation{(\disjunction{\negation{\beta}}{\gamma})}})}$
\end{description}

We now turn to Ciuciura's axiomatization.

\begin{description}
\item[$\ciuciura{1}$] $\di{\alpha}{(\di{\beta}{\alpha})}$
\item[$\ciuciura{2}$] $\di{(\di{\alpha}{(\di{\beta}{\gamma})})}{(\di{(\di{\alpha}{\beta})}{(\di{\alpha}{\gamma})})}$
\item[$\ciuciura{3}$] $\di{(\dc{\alpha}{\beta})}{\alpha}$
\item[$\ciuciura{4}$] $\di{(\dc{\alpha}{\beta})}{\beta}$
\item[$\ciuciura{5}$] $\di{\alpha}{(\di{\beta}{(\dc{\alpha}{\beta})})}$
\item[$\ciuciura{6}$] $\di{\alpha}{(\disjunction{\alpha}{\beta})}$
\item[$\ciuciura{7}$] $\di{\beta}{(\disjunction{\alpha}{\beta})}$
\item[$\ciuciura{8}$] $\di{(\di{\alpha}{\gamma})}{(\di{(\di{\beta}{\gamma})}{\di{(\disjunction{\alpha}{\beta})}{\gamma}})}$
\item[$\ciuciura{9}$] $\disjunction{\alpha}{(\di{\alpha}{\beta})}$
\item[$\ciuciura{10}$] $\negation{(\dc{\negation{\alpha}}{\dc{\negation{\negation{\alpha}}}{\negation{(\disjunction{\alpha}{\negation{\alpha}})}}})}$
\item[$\ciuciura{11}$] $\di{\negation{(\dc{\negation{\alpha}}{\dc{\negation{\beta}}{\negation{(\disjunction{\alpha}{\beta})}}})}}{\negation{(\dc{\negation{\alpha}}{\dc{\negation{\beta}}{\dc{\negation{\gamma}}{\negation{(\disjunction{\alpha}{\disjunction{\beta}{\gamma}})}}}})}}$
\item[$\ciuciura{12}$] $\di{\negation{(\dc{\negation{\alpha}}{\dc{\negation{\beta}}{\dc{\negation{\gamma}}{\negation{(\disjunction{\alpha}{\disjunction{\beta}{\gamma}})}}}})}}{\negation{(\dc{\negation{\alpha}}{\dc{\negation{\gamma}}{\dc{\negation{\beta}}{\negation{(\disjunction{\alpha}{\disjunction{\gamma}{\beta}})}}}})}}$
\item[$\ciuciura{13}$] $\di{\negation{(\dc{\negation{\alpha}}{\dc{\negation{\beta}}{\dc{\negation{\gamma}}{\negation{(\disjunction{\alpha}{\disjunction{\beta}{\gamma}})}}}})}}{(\di{(\disjunction{\alpha}{\disjunction{\beta}{\negation{\gamma}}})}{(\disjunction{\alpha}{\beta})})}$
\item[$\ciuciura{14}$] $\di{\negation{(\dc{\negation{\alpha}}{\negation{\beta}})}}{(\disjunction{\alpha}{\beta})}$
\item[$\ciuciura{15}$] $\di{(\disjunction{\alpha}{(\disjunction{\beta}{\negation{\beta}})})}{\negation{(\dc{\negation{\alpha}}{\negation{(\disjunction{\beta}{\negation{\beta}})}})}}$
\end{description}

\section{Non-equivalence of the axiomatizations}\label{sec:non-equivalence}%

Section~\ref{sec:unprovable-in-c} lists the axioms of $\D$ that are unprovable in $\C$ and Section~\ref{sec:unprovable-in-d} lists the axioms of $\C$ that are unprovable in $\D$.  It is clear that $\C$ and $\D$ share theorems, since some of the axioms of the two (e.g., $\ciuciura{1}$ and $\ddk{1}$) are identical.  Thus, we are in possession of two ``orthogonal'' logics: they overlap but neither theory includes the other.

\subsection{Theorems of $\D$ unprovable in $\C$}\label{sec:unprovable-in-c}

Ciuciura points out that $\ddk{19}$ and $\ddk{22}$ are invalid in the intended semantics of $\dtwo$.  Since Ciuciura's $\C$ is intended to be sound and complete for the intended semantics, it is clear that neither $\ddk{19}$ nor $\ddk{22}$ ought to be theorems of $\C$.  It is worth, however, investigating this from an axiomatic point of view, that is, without referring to the intended semantics.  Moreover, it is worth asking whether there might be other theorems of $\D$ that are unprovable in $\C$.  Are $\ddk{19}$ and $\ddk{22}$ the only ones?  The answer, made clear in the following series of Propositions, is a resounding ``no''.  Briefly, all of $\ddk{10}$ to $\ddk{22}$, save for $\ddk{11}$, are unprovable in $\C$.

\begin{prop}\label{prop:ddk-10-unprovable}$\Cunprovable{\ddk{10}}$\end{prop}

\begin{proof}
Consider the following matrices (Table~\ref{tab:ddk-10-unprovable}) over the three truth values $\{ 1, 2, 3 \}$ for $\discussiveimplicationsymbol$, $\discussiveconjunctionsymbol$, $\disjunctionsymbol$, and $\negationsymbol$:

\begin{table}[h]
\centering\footnotesize
\begin{tabular}{c||c}
Value & Value of $\negationsymbol$\\
\hline
1 & 3\\
2 & 3\\
3 & 2\\
\end{tabular}
\begin{tabular}{c|c||c|c|c|c|}
Value & Value & $\disjunctionsymbol$ & $\discussiveconjunctionsymbol$ & $\discussiveimplicationsymbol$\\
\hline
1 & 1 & 3 & 3 & 3\\
1 & 2 & 1 & 2 & 2\\
1 & 3 & 3 & 1 & 3\\
2 & 1 & 1 & 2 & 3\\
2 & 2 & 2 & 2 & 1\\
2 & 3 & 3 & 2 & 2\\
3 & 1 & 3 & 1 & 1\\
3 & 2 & 3 & 2 & 2\\
3 & 3 & 3 & 3 & 3\\
\end{tabular}
\caption{\label{tab:ddk-10-unprovable}Matrices validating $\C$ in which $\ddk{10}$ is invalid}
\end{table}

We then declare that any formula having the values $1$ or $3$ according to these tables is provable; formulas having value $2$ are unprovable.
\end{proof}

\begin{prop}\label{prop:ddk-12-unprovable}$\Cunprovable{\ddk{12}}$\end{prop}

\begin{proof}
Consider the following matrices (Table~\ref{tab:ddk-12-unprovable}):
\begin{table}[h]
\centering\footnotesize
\begin{tabular}{c||c}
Value & Value of $\negationsymbol$\\
\hline
1 & 4\\
2 & 3\\
3 & 4\\
4 & 1\\
\end{tabular}
\begin{tabular}{c|c||c|c|c|c|}
Value & Value & $\disjunctionsymbol$ & $\discussiveconjunctionsymbol$ & $\discussiveimplicationsymbol$\\
\hline
1 & 1 & 1 & 1 & 4\\
1 & 2 & 2 & 1 & 2\\
1 & 3 & 3 & 1 & 2\\
1 & 4 & 4 & 1 & 4\\
2 & 1 & 2 & 1 & 1\\
2 & 2 & 2 & 3 & 3\\
2 & 3 & 2 & 4 & 3\\
2 & 4 & 2 & 4 & 3\\
3 & 1 & 3 & 1 & 1\\
3 & 2 & 3 & 3 & 4\\
3 & 3 & 3 & 2 & 4\\
3 & 4 & 4 & 2 & 3\\
4 & 1 & 4 & 1 & 1\\
4 & 2 & 4 & 4 & 2\\
4 & 3 & 4 & 3 & 3\\
4 & 4 & 4 & 4 & 3\\
\end{tabular}
\caption{\label{tab:ddk-12-unprovable}Matrices validating $\C$ in which $\ddk{12}$ is invalid}
\end{table}
We then declare that any formula having the values $2$, $3$, or $4$ according to these tables is provable; formulas having value $1$ declared to be unprovable.
\end{proof}

\begin{prop}\label{prop:ddk-13-unprovable}$\Cunprovable{\ddk{13}}$\end{prop}

\begin{proof}
Consider the following matrices (Table~\ref{tab:ddk-13-unprovable}):
\begin{table}[h]
\centering\footnotesize
\begin{tabular}{c||c}
Value & Value of $\negationsymbol$\\
\hline
1 & 3\\
2 & 2\\
3 & 1\\
\end{tabular}
\begin{tabular}{c|c||c|c|c|c|}
Value & Value & $\disjunctionsymbol$ & $\discussiveconjunctionsymbol$ & $\discussiveimplicationsymbol$\\
\hline
1 & 1 & 2 & 1 & 2\\
1 & 2 & 1 & 1 & 1\\
1 & 3 & 1 & 3 & 3\\
2 & 1 & 2 & 2 & 1\\
2 & 2 & 1 & 2 & 1\\
2 & 3 & 2 & 3 & 3\\
3 & 1 & 1 & 3 & 1\\
3 & 2 & 2 & 3 & 1\\
3 & 3 & 3 & 3 & 1\\
\end{tabular}
\caption{\label{tab:ddk-13-unprovable}Matrices validating $\C$ in which $\ddk{13}$ is invalid}
\end{table}
We then declare that any formula having the values $1$ or $2$ according to these tables is provable; formulas having value $3$ are declared to be unprovable.
\end{proof}

\begin{prop}\label{prop:ddk-14-unprovable}$\Cunprovable{\ddk{14}}$\end{prop}

\begin{proof}
Consider the following matrices (Table~\ref{tab:ddk-14-unprovable}):
\begin{table}[h]
\centering\footnotesize
\begin{tabular}{c||c}
Value & Value of $\negationsymbol$\\
\hline
1 & 3\\
2 & 1\\
3 & 1\\
\end{tabular}
\begin{tabular}{c|c||c|c|c|c|}
Value & Value & $\disjunctionsymbol$ & $\discussiveconjunctionsymbol$ & $\discussiveimplicationsymbol$\\
\hline
1 & 1 & 2 & 1 & 1\\
1 & 2 & 2 & 1 & 1\\
1 & 3 & 1 & 3 & 3\\
2 & 1 & 2 & 1 & 1\\
2 & 2 & 2 & 2 & 2\\
2 & 3 & 2 & 3 & 3\\
3 & 1 & 1 & 3 & 1\\
3 & 2 & 2 & 3 & 1\\
3 & 3 & 3 & 3 & 1\\
\end{tabular}
\caption{\label{tab:ddk-14-unprovable}Matrices validating $\C$ in which $\ddk{14}$ is invalid}
\end{table}
We then declare that any formula having the values $1$ or $2$ according to these tables is provable; formulas having value $3$ are declared to be unprovable.
\end{proof}

\begin{prop}\label{prop:ddk-15-unprovable}$\Cunprovable{\ddk{15}}$\end{prop}

\begin{proof}
Consider the following matrices (Table~\ref{tab:ddk-15-unprovable}):
\begin{table}[h]
\centering\footnotesize
\begin{tabular}{c|c||c|c|c|c|}
Value & Value & $\disjunctionsymbol$ & $\discussiveconjunctionsymbol$ & $\discussiveimplicationsymbol$\\
\hline
1 & 1 & 3 & 1 & 1\\
1 & 2 & 1 & 2 & 2\\
1 & 3 & 3 & 1 & 3\\
2 & 1 & 1 & 2 & 3\\
2 & 2 & 2 & 2 & 1\\
2 & 3 & 3 & 2 & 3\\
3 & 1 & 1 & 3 & 1\\
3 & 2 & 3 & 2 & 2\\
3 & 3 & 1 & 3 & 3\\
\end{tabular}
\caption{\label{tab:ddk-15-unprovable}Matrices validating $\C$ in which $\ddk{15}$ is invalid.}
\end{table}
The matrix for negation is the same as the one appearing in Table~\ref{tab:ddk-10-unprovable}.  We then declare that any formula having the values $1$ or $3$ according to these tables is provable; formulas having value $2$ are declared to be unprovable.
\end{proof}

\begin{prop}\label{prop:ddk-16-unprovable}$\Cunprovable{\ddk{16}}$\end{prop}

\begin{proof}
Consider the following matrices (Table~\ref{tab:ddk-16-unprovable}):
\begin{table}[h]
\centering\footnotesize
\begin{tabular}{c||c}
Value & Value of $\negationsymbol$\\
\hline
1 & 3\\
2 & 2\\
3 & 1\\
4 & 4\\
\end{tabular}
\begin{tabular}{c|c||c|c|c|c|}
Value & Value & $\disjunctionsymbol$ & $\discussiveconjunctionsymbol$ & $\discussiveimplicationsymbol$\\
\hline
1 & 1 & 1 & 1 & 4\\
1 & 2 & 2 & 1 & 2\\
1 & 3 & 4 & 1 & 2\\
1 & 4 & 4 & 1 & 2\\
2 & 1 & 3 & 1 & 1\\
2 & 2 & 4 & 2 & 4\\
2 & 3 & 3 & 2 & 2\\
2 & 4 & 2 & 4 & 2\\
3 & 1 & 4 & 1 & 1\\
3 & 2 & 3 & 3 & 2\\
3 & 3 & 3 & 3 & 2\\
3 & 4 & 2 & 2 & 4\\
4 & 1 & 4 & 1 & 1\\
4 & 2 & 2 & 2 & 4\\
4 & 3 & 4 & 2 & 3\\
4 & 4 & 4 & 2 & 3\\
\end{tabular}
\caption{\label{tab:ddk-16-unprovable}Matrices validating $\C$ in which $\ddk{16}$ is invalid}
\end{table}
We then declare that any formula having the values $2$, $3$, or $4$ according to these tables is provable; formulas having value $1$ declared to be unprovable.
\end{proof}

\begin{prop}\label{prop:ddk-17-unprovable}$\Cunprovable{\ddk{17}}$\end{prop}

\begin{proof}
Consider the following matrices (Table~\ref{tab:ddk-17-unprovable}):
\begin{table}[h]
\centering\footnotesize
\begin{tabular}{c|c||c|c|c|c|}
Value & Value & $\disjunctionsymbol$ & $\discussiveconjunctionsymbol$ & $\discussiveimplicationsymbol$\\
\hline
1 & 1 & 2 & 1 & 1\\
1 & 2 & 1 & 1 & 1\\
1 & 3 & 1 & 3 & 3\\
2 & 1 & 1 & 1 & 1\\
2 & 2 & 1 & 2 & 1\\
2 & 3 & 2 & 3 & 3\\
3 & 1 & 1 & 3 & 2\\
3 & 2 & 2 & 3 & 1\\
3 & 3 & 3 & 3 & 1\\
\end{tabular}
\caption{\label{tab:ddk-17-unprovable}Matrices validating $\C$ in which $\ddk{17}$ is invalid}
\end{table}
The matrix for negation is the same as the one used in Table~\ref{tab:ddk-13-unprovable}. We then declare that any formula having the values $1$ or $2$ according to these tables is provable; formulas having value $3$ are declared to be unprovable.
\end{proof}

\begin{prop}\label{prop:ddk-18-unprovable}$\Cunprovable{\ddk{18}}$\end{prop}

\begin{proof}
Consider the following matrices (Table~\ref{tab:ddk-18-unprovable}):
\begin{table}[h]
\centering\footnotesize
\begin{tabular}{c|c||c|c|c|c|}
Value & Value & $\disjunctionsymbol$ & $\discussiveconjunctionsymbol$ & $\discussiveimplicationsymbol$\\
\hline
1 & 1 & 2 & 1 & 1\\
1 & 2 & 2 & 2 & 1\\
1 & 3 & 1 & 3 & 3\\
2 & 1 & 1 & 2 & 1\\
2 & 2 & 1 & 1 & 1\\
2 & 3 & 2 & 3 & 3\\
3 & 1 & 1 & 3 & 2\\
3 & 2 & 2 & 3 & 2\\
3 & 3 & 3 & 3 & 2\\
\end{tabular}
\caption{\label{tab:ddk-18-unprovable}Matrices validating $\C$ in which $\ddk{18}$ is invalid.  The matrix for negation is the same as the one appearing in Table~\ref{tab:ddk-14-unprovable}.}
\end{table}
We then declare that any formula having the values $1$ or $2$ according to these tables is provable; formulas having value $3$ are declared to be unprovable.
\end{proof}

\begin{prop}\label{prop:ddk-19-unprovable}$\Cunprovable{\ddk{19}}$\end{prop}

\begin{proof}
Consider the following matrices (Table~\ref{tab:ddk-19-unprovable}):
\begin{table}[h]
\centering\footnotesize
\begin{tabular}{c|c||c|c|c|c|}
Value & Value & $\disjunctionsymbol$ & $\discussiveconjunctionsymbol$ & $\discussiveimplicationsymbol$\\
\hline
1 & 1 & 2 & 1 & 1\\
1 & 2 & 2 & 1 & 1\\
1 & 3 & 1 & 3 & 3\\
2 & 1 & 1 & 2 & 2\\
2 & 2 & 1 & 2 & 2\\
2 & 3 & 2 & 3 & 3\\
3 & 1 & 1 & 3 & 2\\
3 & 2 & 2 & 3 & 1\\
3 & 3 & 3 & 3 & 1\\
\end{tabular}
\caption{\label{tab:ddk-19-unprovable}Matrices validating $\C$ in which $\ddk{19}$ is invalid. The matrix for negation is the same as the one appearing in Table~\ref{tab:ddk-14-unprovable}.}
\end{table}
We then declare that any formula having the values $1$ or $2$ according to these tables is provable; formulas having value $3$ are declared to be unprovable.
\end{proof}

\begin{prop}\label{prop:ddk-20-unprovable}$\Cunprovable{\ddk{20}}$\end{prop}

\begin{proof}
Consider the following matrices (Table~\ref{tab:ddk-20-unprovable}):
\begin{table}[h]
\centering\footnotesize
\begin{tabular}{c||c}
Value & Value of $\negationsymbol$\\
\hline
1 & 3\\
2 & 3\\
3 & 1\\
\end{tabular}
\begin{tabular}{c|c||c|c|c|c|}
Value & Value & $\disjunctionsymbol$ & $\discussiveconjunctionsymbol$ & $\discussiveimplicationsymbol$\\
\hline
1 & 1 & 1 & 1 & 2\\
1 & 2 & 2 & 1 & 3\\
1 & 3 & 3 & 1 & 2\\
2 & 1 & 2 & 1 & 1\\
2 & 2 & 3 & 2 & 3\\
2 & 3 & 3 & 2 & 2\\
3 & 1 & 3 & 1 & 1\\
3 & 2 & 3 & 2 & 2\\
3 & 3 & 2 & 3 & 2\\
\end{tabular}
\caption{\label{tab:ddk-20-unprovable}Matrices validating $\C$ in which $\ddk{20}$ is invalid}
\end{table}
We then declare that any formula having the values $2$ or $3$ according to these tables is provable; formulas having value $1$ are declared to be unprovable.
\end{proof}

\begin{prop}\label{prop:ddk-21-unprovable}$\Cunprovable{\ddk{21}}$\end{prop}

\begin{proof}
Consider the following matrices (Table~\ref{tab:ddk-21-unprovable}):
\begin{table}[h]
\centering\footnotesize
\begin{tabular}{c|c||c|c|c|c|}
Value & Value & $\disjunctionsymbol$ & $\discussiveconjunctionsymbol$ & $\discussiveimplicationsymbol$\\
\hline
1 & 1 & 2 & 1 & 2\\
1 & 2 & 2 & 1 & 1\\
1 & 3 & 1 & 3 & 3\\
2 & 1 & 2 & 1 & 2\\
2 & 2 & 1 & 2 & 1\\
2 & 3 & 2 & 3 & 3\\
3 & 1 & 1 & 3 & 2\\
3 & 2 & 2 & 3 & 1\\
3 & 3 & 3 & 3 & 2\\
\end{tabular}
\caption{\label{tab:ddk-21-unprovable}Matrices validating $\C$ in which $\ddk{21}$ is invalid.  The matrix for negation is the same as the one appearing in Table~\ref{tab:ddk-13-unprovable}.}
\end{table}
We then declare that any formula having the values $1$ or $2$ according to these tables is provable; formulas having value $3$ are declared to be unprovable.
\end{proof}

\begin{prop}\label{prop:ddk-22-unprovable}$\Cunprovable{\ddk{22}}$\end{prop}

\begin{proof}
Consider the following matrices (Table~\ref{tab:ddk-22-unprovable}):
\begin{table}[h]
\centering\footnotesize
\begin{tabular}{c|c||c|c|c|c|}
Value & Value & $\disjunctionsymbol$ & $\discussiveconjunctionsymbol$ & $\discussiveimplicationsymbol$\\
\hline
1 & 1 & 2 & 1 & 2\\
1 & 2 & 1 & 2 & 1\\
1 & 3 & 1 & 3 & 3\\
2 & 1 & 1 & 1 & 2\\
2 & 2 & 1 & 1 & 1\\
2 & 3 & 2 & 3 & 3\\
3 & 1 & 1 & 3 & 2\\
3 & 2 & 2 & 3 & 1\\
3 & 3 & 3 & 3 & 2\\
\end{tabular}
\caption{\label{tab:ddk-22-unprovable}Matrices validating $\C$ in which $\ddk{22}$ is invalid.  The matrix for negation is the same as the one appearing in Table~\ref{tab:ddk-13-unprovable}.}
\end{table}
We then declare that any formula having the values $1$ or $2$ according to these tables is provable; formulas having value $3$ are declared to be unprovable.
\end{proof}

\subsection{$\Sfive$-(in)validities among the $\D$ axioms}\label{sec:}

Section~\ref{sec:unprovable-in-c} presents several axiom (schemes) of $\D$ cannot be proved in $\C$ (that is, for each of several of $\D$'s axiom schemes, at least one instance of the scheme is not provable in $\C$).  As pointed out by Ciuciura, the fact that some of these theorems are invalid according to the intended semantics shows that their $\C$-unprovability is to be expected (with the understanding, of course, that $\C$ is supposed to be an axiomatization of $\dtwo$).  But it is too quick to conclude that this diagnosis applies to all of the $\C$-unprovable axioms of $\D$.  Indeed, some of sentences appearing in Proposition~\ref{prop:ddk-10-unprovable}--\ref{prop:ddk-22-unprovable} are in fact valid in the intended semantics of $\dtwo$.  Table~\ref{tab:valid-or-invalid} separates the two.

\begin{table}[h]
\centering\footnotesize
\begin{tabular}{|c|c|}
$\D$-axiom & $\dtwo$-valid?\\
\hline
$\ddk{10}$ & + \\
$\ddk{12}$ & - \\
$\ddk{13}$ & - \\
$\ddk{14}$ & - \\
$\ddk{15}$ & + \\
$\ddk{16}$ & - \\
$\ddk{17}$ & + \\
$\ddk{18}$ & ? \\
$\ddk{19}$ & - \\
$\ddk{20}$ & + \\
$\ddk{21}$ & ? \\
$\ddk{22}$ & - \\
\end{tabular}
\caption{\label{tab:valid-or-invalid}$\Sfive$ validities and invalidities among the $\D$ axioms.  ``?'' indicates that we were unable to settle the validity question, owing to the complexity of the (translated) formulas.}
\end{table}

\begin{prop}\label{prop:s5}
If $\phi$ and $\psi$ are $\Sfive$-valid formulas such that $\equivalence{\phi}{\psi}$ is a tautology, then $\possibly{(\implication{\possibly{\phi}}{\psi})}$ is $\Sfive$-valid.
\end{prop}

\begin{proof}
Consider an $\Sfive$-model $\mathcal{M} = (W,R,V)$ (a triple of a non-empty set of worlds, an equivalence relation, and a function from atoms to sets of worlds) and a world $w$ in $W$.  If $\falsifies{u}{\psi}$ for every world $u$ accessible from $w$, then (since $R$ is reflexive) we have
\[
\truthifies{w}{\possibly{(\implication{\possibly{\phi}}{\psi})}}
\]
trivially.  Suppose, then, that $u$ is a world accessible from $w$ for which $\truthifies{u}{\psi}$.  Then, since $\equivalence{\phi}{\psi}$ is a tautology, $\truthifies{u}{\phi}$ as well.  Since $R$ is reflexive, we have that $\truthifies{u}{\possibly{\phi}}$ as well as $\truthifies{u}{\implication{\possibly{\phi}}{\psi}}$.  Thus
\[
\truthifies{w}{\possibly{(\implication{\possibly{\phi}}{\psi})}}
\]
as desired.
\end{proof}

The standard translation $\tau$ from $\dtwo$-formulas (the connectives are $\negationsymbol$, $\disjunctionsymbol$ as well as $\discussiveimplicationsymbol$ and $\discussiveconjunctionsymbol$), to standard (non-discussive) modal formulas (built from atoms and the usual battery $\negationsymbol$, $\implicationsymbol$, $\disjunctionsymbol$, $\conjunctionsymbol$) is standardly defined as not changing $\disjunctionsymbol$ and $\negationsymbol$.  Proposition~\ref{prop:s5} thus yields:

\begin{coro}If $\dtwo$-formulas $\phi$ and $\psi$ contain neither $\discussiveimplicationsymbol$ nor $\discussiveconjunctionsymbol$ (that is, both are built from atoms and $\disjunctionsymbol$ and $\negationsymbol$ alone), and if $\equivalence{\phi}{\psi}$ is a tautology, then $\di{\phi}{\psi}$ is $\dtwo$-valid.
\end{coro}
\begin{proof}
$\di{\phi}{\psi}$ is $\dtwo$-valid just in case $\tau{(\di{\phi}{\psi})}$ is $\Sfive$-valid.  Since the discussive connectives $\discussiveconjunctionsymbol$ and $\discussiveimplicationsymbol$ appear neither in $\phi$ nor in $\psi$, we are to prove that
\[
\possibly{(\implication{\possibly{\phi}}{\psi})}
\]
is $\Sfive$-valid.  This follows from Proposition~\ref{prop:s5}.
\end{proof}

The corollary allows us to quickly identify $\ddk{10}$, $\ddk{15}$, $\ddk{17}$, and $\ddk{20}$ as $\dtwo$-valid, since these formulas are discussive implications whose antecedents and consequents are do not involve the $\discussiveimplicationsymbol$ and $\discussiveconjunctionsymbol$ connectives.  For example, $\ddk{10}$ is dispatched because $\alpha$ and $\negation{\negation{\alpha}}$ are tautologically equivalent.  $\ddk{18}$  and $\ddk{21}$ are not covered by the corollary, since their antecedents and consequents involve discussive connectives.

\subsection{Theorems of $\C$ unprovable in $\D$}\label{sec:unprovable-in-d}

(The following section is a digression from the main line of investigation but is included for the sake of completeness.)

Section~\ref{sec:unprovable-in-c} lists many axioms of $\D$ that are unprovable in $\D$.  Ciuciura singled out $\ddk{19}$ and $\ddk{22}$ as invalid in the intended semantics for $\dtwo$, and thus ought to be missing from an axiomatization.  One might suspect that the way out is to restrict $\D$: throw out the axioms that are not valid in the intended semantics, but keep the rest (possibly reformulated in a more perspicuous way).  On this line of reasoning, one expects to find that every axiom of $\C$ is provable in $\D$ (though one hopes that the axioms $\ddk{19}$ and $\ddk{22}$ singled out by Ciuciura are avoided).  But this is not so: one finds that Ciuciura's axiomatization $\C$ is not simply a restriction of $\D$.  It turns out that (at least) one axiom of $\C$ is unprovable in $\D$.

\begin{prop}\label{prop:ciuciura-13-unprovable}$\Dunprovable{\ciuciura{13}}$\end{prop}

\begin{proof}
Consider the following matrices (Table~\ref{tab:c-13-unprovable}):
\begin{table}[h]
\centering
\begin{tabular}{c||c}
Value & Value of $\negationsymbol$\\
\hline
1 & 2\\
2 & 1\\
3 & 4\\
4 & 3\\
\end{tabular}
\begin{tabular}{c|c||c|c|c|c|}
Value & Value & $\disjunctionsymbol$ & $\discussiveconjunctionsymbol$ & $\discussiveimplicationsymbol$\\
\hline
1 & 1 & 1 & 1 & 2\\
1 & 2 & 2 & 1 & 2\\
1 & 3 & 3 & 1 & 2\\
1 & 4 & 4 & 1 & 2\\
2 & 1 & 2 & 1 & 1\\
2 & 2 & 2 & 2 & 2\\
2 & 3 & 2 & 3 & 3\\
2 & 4 & 2 & 4 & 4\\
3 & 1 & 3 & 1 & 1\\
3 & 2 & 2 & 2 & 2\\
3 & 3 & 3 & 3 & 3\\
3 & 4 & 2 & 4 & 4\\
4 & 1 & 4 & 1 & 1\\
4 & 2 & 2 & 2 & 2\\
4 & 3 & 2 & 3 & 3\\
4 & 4 & 4 & 4 & 4\\
\end{tabular}
\caption{\label{tab:c-13-unprovable}Matrices validating $\D$ in which $\ciuciura{13}$ is invalid}
\end{table}
We then declare that any formula having the values $2$, $3$, or $4$ according to these tables is provable; formulas having value $1$ declared to be unprovable.
\end{proof}

\section{Conclusion and further problems}\label{sec:prospects}

In light of the results of Section~\ref{sec:non-equivalence}, it would appear that the Ciuciura's completeness theorem (Theorem 4.3 of~\cite{ciuciura2008frontiers}) is mistaken.  We thus seem to lack a complete axiomatization of $\dtwo$.  More accurately, neither $\C$ nor $\D$ is a complete axiomatization of $\dtwo$.  Other axiomatizations, such Vasyukov's~\cite{vasyukov2001new}, deserves to be compared to $\C$ and $\D$.  It is correct to point out that some of the non-theorems of $\D$ are indeed not valid in the intended semantics, so they ought not to be added to $\C$.  However, some $\D$-theorems that are unprovable in $\C$, such as $\ddk{10}$ and $\ddk{11}$, are indeed valid in the intended semantics for $\dtwo$.  It would thus appear that the challenge of axiomatizing $\dtwo$ remains open.  The most problematic formulas are those which~(i)~are valid in the intended semantics for~$\dtwo$ but which~(ii)~are not theorems of $\C$.  These are collected in Table~\ref{tab:valid-or-invalid}.  Does one obtain an axiomatization of $\dtwo$ by imply adding these to $\C$?

\paragraph{Acknowledgments.}
The author was supported by FWF grant P25417-G15 (LOGFRADIG).

\bibliographystyle{sl}
\bibliography{jaskowski}

\begin{thebibliography}{1}

\bibitem{achtelik1981independence}
\textsc{Achtelik, G.}, \textsc{L.~Dubikajtis}, \textsc{E.~Dudek}, and
  \textsc{J.~Kanior}, `On independence of axioms in ja{\'s}kowski discussive
  propositional calculus', \emph{Reports on Mathematical Logic}, 11 (1981),
  3--11.

\bibitem{ciuciura2008frontiers}
\textsc{Ciuciura, Janusz}, `Frontiers of the discursive logic', \emph{Bulletin
  of the Section of Logic}, 37 (2008), 2, 81--92.

\bibitem{dacosta1977new}
\textsc{da~Costa, N. C.~A.}, and \textsc{L.~Dubikajtis}, `New axiomatization
  for the discussive propositional calculus', in A.~I. Arruda, N.~C.~A.
  da~Costa, and R.~Chuaqui, (eds.), \emph{Non Classical Logics, Model Theory
  and Computability}, North-Holland, 1977, pp. 45--55.

\bibitem{kotas1977modal}
\textsc{Kotas, J.}, and \textsc{N.~C.~A. da~Costa}, `On some modal logical
  systems defined in connexion with ja{\'s}kowski's problem', in A.~I. Arruda,
  N.~C.~A. da~Costa, and R.~Chuaqui, (eds.), \emph{Non Classical Logics, Model
  Theory and Computability}, North-Holland, 1977, pp. 57--73.

\bibitem{vasyukov2001new}
\textsc{Vasyukov, Vladimir~L.}, `A new axiomatization of jaskowski's discussive
  logic', \emph{Logic and Logical Philosophy}, 9 (2001), 35--46.

\end{thebibliography}


\AuthorAdressEmail{Jesse Alama}{Theory and Logic Group\\
Technical University of Vienna\\
Vienna, Austria}{alama@logic.at}

\end{document}